\documentclass[a4paper,12pt,onecolumn]{article}
\usepackage{etex}
\usepackage[vmargin=2cm,hmargin=2cm,headheight=14.5pt,top=2cm,headsep=.5cm]{geometry}

\usepackage{bm}
\usepackage{empheq}
\usepackage{stackrel}
\usepackage{cases}
\usepackage{mathtools}
\usepackage{amsthm,amsmath,amscd}
\usepackage{amssymb,amsfonts}
\usepackage{makeidx}
\usepackage[all]{xy}
\usepackage{perpage}
\usepackage{tikz-cd}
\tikzcdset{every label/.append style = {font = \small}}
\usepackage[symbol]{footmisc}

\MakePerPage[2]{footnote}
\usepackage{graphicx}
\DeclareMathSizes{12}{12}{8}{6}

\usepackage{cite}
\usepackage{url}
\usepackage[charter]{mathdesign}
\usepackage{accents}

\newtheoremstyle{ptheorem}{1em}{0em}{\itshape}{}{\bfseries}{.}{.5em}{\thmname{#1}\thmnumber{ #2}\thmnote{ (\hspace{-.01pt}{#3})}}

\theoremstyle{ptheorem}

\newtheorem{thm}{Theorem}[section]

\newtheorem{lem}[thm]{Lemma}

\newtheoremstyle{hdef}{1em}{0em}{}{}{\bfseries}{.}{.5em}{\thmname{#1}\thmnumber{ #2}\thmnote{ (\hspace{-.01pt}{#3})}}
\theoremstyle{hdef}

\newtheorem{rem}[thm]{Remark}

\makeatletter
\newtheoremstyle{premark}{1em}{0em}{
\addtolength{\@totalleftmargin}{1.5em}
\addtolength{\linewidth}{-1.5em}
\parshape 1 1.5em \linewidth}{}{\scshape}{.}{.5em}{}
\makeatother

\theoremstyle{premark}

\numberwithin{equation}{section}
\numberwithin{figure}{section}

\renewcommand{\phi}{\varphi}

\newcommand{\olb}[1]{%
  \vbox{\offinterlineskip\ialign{\hfil##\hfil\cr $\rotatebox[origin=c]{90}{$]$}$\cr\noalign{\kern-.45ex}{$#1$}\cr}}}
\allowdisplaybreaks
\begin{document}

\title{A Hyperbolic Analog of the Phasor Addition Formula\footnote{The author would like to acknowledge his gratitude to Santiago Codesido for suggesting the relation of this paper with critical phenomena, to Prof. Alberto Cabada for encouraging the author to write this paper and his remarkable suggestions, and to the anonymous referee for her or his valuable contributions which helped improve paper, specially the last section. The author was partially supported, while writing this article, by FEDER and Ministerio de Educaci\'on y Ciencia, Spain, project MTM2010-15314 and a FPU scholarship, Ministerio de Educaci\'on, Cultura y Deporte, Spain.}
}
\markright{Hyperbolic Phasor Addition Formula}

\author{F. Adri\'an F. Tojo}

\maketitle

\begin{abstract}
In this article we review the basics of the phasor formalism in a rigorous way, highlighting the physical motivation behind it and presenting a hyperbolic counterpart of the phasor addition formula.
\end{abstract}

\section{Introduction.}
The idea for this paper was born when the author was confronted with the need of simplifying linear combinations of hyperbolic sines and cosines with the same argument into a single trigonometric expression in order to solve for that argument (see \cite{Cab}). In the usual euclidean case, there are very well know formulae for the sum of linear combinations of sines and cosines. In particular, we have the phasor addition formula (equations \eqref{eqpaf}--\eqref{eqpafc} are some of its incarnations) which, somehow, is a generalization of the standard formula $\cos x+\sin x=\sqrt 2\sin(x+\pi/4)$. Nevertheless, similar formulae for the hyperbolic case seem to be absent from the literature.\par
It is interesting to note that something that seems so trivial as a mere algebraic manipulation has profound (and very well studied) roots in physics, where these linear combinations (in the euclidean case) occur naturally when studying \textbf{phasors}. This paper is written with the intention of introducing the reader to the usual phasor formalism used in physics and the motivation behind it, containing all the rigor expected by a mathematician. It will also generalize the formulae previously derived for the hyperbolic case with the hope they may eventually become handy for the reader.
\section{Phasors in physics.}
To be more precise, phasors appear in Physics from the need of establishing some kind of arithmetic for the set of functions
 $$\mathcal F:=\{f:\mathbb R\to\mathbb R\ :\ f(t)=a\cos(\omega t+\varphi),\ a\in\mathbb R,\ \varphi\in\mathbb R|_\sim\},$$
for some fixed $\omega\in\mathbb R\backslash\{0\}$ and where $\varphi_1\sim\varphi_2$ if and only if $ \varphi_1-\varphi_2\in 2\pi\mathbb Z$ for any $\varphi_1,\varphi_2\in\mathbb R$.
The parameters present in the functions of $\mathcal F$ are called, respectively, \textbf{amplitude} ($a$), \textbf{frequency} ($\omega$) and \textbf{phase} ($\varphi$). The functions in $\mathcal F$ occur mostly in problems related to Mechanics and Electronics (see, for instance, \cite{Des, Pha, Ser}), but their origin is rooted in arguably the most important problem in Physics: the \textit{harmonic oscillator}.\par
If we consider one space variable $x$ and a time variable $t$, the \textit{Euler-Lagrange equation} of motion (a fundamental principle of Dynamics) implies that the equation of motion of a free particle is given by
\begin{equation}\label{eqNew}
m\,x''(t)+V'(x(t))=0,
\end{equation}
where $m$ is the mass of the particle and $V:\mathbb R\to\mathbb R$ is a given potential. Equation \eqref{eqNew} is, basically, \textit{Newton's second Law of motion} for the potential $V$.\par
In many problems of Physics it is common to chose as potential a quadratic function of the kind $V(x)=\frac{1}{2}k\,x^2$ with $k>0$. This is the case, for instance, of \textit{Hook's Law} on the force of a spring, but this kind of potential also occurs in problems concerning pendula (when the angle of displacement is considered to be small),  RLC circuits, or acoustical systems. If fact, this potential appears naturally when taking a `first order' approximation for small perturbations on a mass in a stable equilibrium with respect to the forces it is subject to.\par
Hence, considering $V(x)=\frac{1}{2}k\,x^2$, and defining $\omega=\sqrt{k/m}$, we have that equation \eqref{eqNew} can be expressed as
\begin{equation}\label{eqho}x''(t)+\omega^2x(t)=0,\end{equation}
which is known as the equation of the \textbf{harmonic oscillator}.
The set of solutions of this equation is precisely
$$\{a\cos\omega t+b\cos\left(\omega t+\pi/2\right)\ :\ a,b\in\mathbb R\}$$
(observe that $-\cos\left(\omega t+\pi/2\right)=\sin\omega t$). Therefore, the need for adding functions in $\mathcal F$ appears in a natural way, because they are the solutions of one or more harmonic oscillators with the same constant $k$.\par
Now, the question that almost any mathematician would ask is, `what happens when $k<0$?' This situation has to do with the theory of critical phenomena \cite{Sta}. Briefly speaking, the potential has a critical point at $k=0$ and for $k<0$ the physical laws change qualitatively. This is the case of phase transitions in matter, for instance, the change from liquid to vapor or from being a normal conductor to being a superconductor.\par
In this new scenario, we can define $\omega=\sqrt{-k/m}$ and the equation derived from equation \eqref{eqNew} is
\begin{equation}\label{eqhho}x''(t)-\omega^2x(t)=0,\end{equation}
which has
$$\{a\cosh\omega t+b\sinh\omega t\ :\ a,b\in\mathbb R\}$$
as set of solutions. Now, can we develop a hyperbolic version of the phasor understanding of equation \eqref{eqhho}? Section 4 will answer this question and in Section 3 we establish the basics of the phasor formalism. Finally, Section 5 is a brief note on the possible extensions of the phasor addition formula and a new way of obtaining it.
\section{The phasor addition formula.}
Fix $\omega\in\mathbb R$. First of all, we will show that $\mathcal F$ is a group using some basic group algebra. Let
$$\mathcal F_\mathbb C:=\{f:\mathbb R\to\mathbb C\ :\ f(t)=ze^{i\omega t},\,z\in\mathbb C\}.$$
The functions in $\mathcal F$ are called \textbf{phasors}. Observe that the map $P:\mathcal C(\mathbb R,\mathbb C)\to\mathcal C(\mathbb R,\mathbb C)$ such that  $Pf(t)=f(t)/e^{i\omega t}$ is a group isomorphism with respect to the sum. We have that $\mathcal F$ is a subset of $\mathcal C(\mathbb R,\mathbb C)$ and $(\mathbb C,+)$, identified with the set of constant functions of $\mathcal C(\mathbb R,\mathbb C)$, is a subgroup of $\mathcal C(\mathbb R,\mathbb C)$. Furthermore, $P|_{\mathcal F_\mathbb C}:\mathcal F_{\mathbb C}\to\mathbb C$ is bijective. Hence, $(\mathcal F_\mathbb C,+)$ is a group. To see this it suffices to see that $x+y\in\mathcal F_\mathbb C$ for any given $x,y\in\mathcal F_\mathbb C$. $P(x),P(y)\in\mathbb C$ and, since $(\mathbb C,+)$ is a group, $P(x)+P(y)\in \mathbb C$. Thus, $P^{-1}(P(x)+P(y))=x+y\in\mathcal F_\mathbb C$.\par

On the other hand, consider the real part operator $\Re:\mathcal C(\mathbb R,\mathbb C)\to\mathcal C(\mathbb R,\mathbb R)$. $\Re$ is a surjective homomorphism and $\Re|_{\mathcal F_{\mathbb C}}:\mathcal F_\mathbb C\to\mathcal F$ is a surjective function. Thus, $\mathcal F$ is also a group. To see this, let $x,y\in\mathcal F$ and $x',y'\in\mathcal F_\mathbb C$ such that $\Re(x')=x,\Re(y')=y$. Hence, $x'+y'\in\mathcal F_\mathbb C$ and $\Re(x'+y')=x+y\in\mathcal F$. Due to these homomorphisms between the considered groups, to study the sum in $\mathcal F$, it is enough to study the sum in $\mathbb C$.\par
Let $a\,e^{i\varphi}$, $b\,e^{i\psi}\in\mathbb C\backslash\{0\}$. Then $a\,e^{i\varphi}+b\,e^{i\psi}=c\,e^{i\theta}$ for some $c\in\mathbb R^+$ and $\theta\in\mathbb R|_\sim$. Observe that\par
$$a\,e^{i\varphi}=a\cos\varphi+ia\sin\varphi,\quad b\,e^{i\psi}=b\cos\psi+ib\sin\psi,$$
so
$$a\,e^{i\varphi}+b\,e^{i\psi}=a\cos\varphi+b\cos\psi+i(a\sin\varphi+b\sin\psi).$$
Therefore, using the law of cosines,
\begin{align*}c & =|a\,e^{i\varphi}+b\,e^{i\psi}|=\sqrt{(a\cos\varphi+b\cos\psi)^2+(a\sin\varphi+b\sin\psi)^2}\\ & =\sqrt{a^2+b^2+2\,a\,b\cos(\varphi-\psi)}.\end{align*}
In order to get $\theta$, we consider the principal argument function $\operatorname{arg}$\footnote{The principal argument function is basically the $\texttt{atan2}$ function common to the math libraries of many computer languages such as FORTRAN \cite[p. 42]{Org}, C, Java, Python, Ruby or Pearl. The principal advantage of having two arguments instead of one, unlike in the traditional definition of the $\arctan$ function, is that it returns the appropriate quadrant of the angle, something that cannot be achieved with the $\arctan$. Some more basic information on the $\texttt{atan2}$ function and its usage can be found at \url{http://en.wikipedia.org/wiki/Atan2}.} such that, for every $z=x+iy\in\mathbb C$, $\operatorname{arg}(z)=\alpha$ where $\alpha$ is the only angle in $[-\pi,\pi)$ satisfying $\sin\alpha=y/\sqrt{x^2+y^2}$ and $\cos\alpha=x/\sqrt{x^2+y^2}$.

Therefore, $\theta=\operatorname{arg}(a\cos\varphi+b\cos\psi+i(a\sin\varphi+b\sin\psi))$. So we can conclude that
\begin{equation}\label{eqpaf}a\,e^{i\varphi}+b\,e^{i\psi}=\sqrt{a^2+b^2+2\,a\,b\cos(\varphi-\psi)}e^{i\operatorname{arg}(a\cos\varphi+b\cos\psi+i(a\sin\varphi+b\sin\psi))}.\end{equation}
Equation \eqref{eqpaf} is called the \textbf{phasor addition formula}.\par
If we want to write equation \eqref{eqpaf} in terms of the elements of $\mathcal F$, we just have to take the real part on both sides of the equation:

\begin{multline}\label{eqpaf2}a\cos(\omega t+\varphi)+b\cos(\omega t+\psi)= \sqrt{a^2+b^2+2\,a\,b\cos(\varphi-\psi)}\cdot\\\cos[\omega t+\operatorname{arg}(a\cos\varphi+b\cos\psi+i(a\sin\varphi+b\sin\psi))].\end{multline}
In particular,
\begin{equation}\begin{aligned}\label{eqpafc}& a\cos(\omega t)+b\sin(\omega t)=a\cos(\omega t)-b\cos(\omega t+\pi/2) \\ = & \sqrt{a^2+b^2}\cos[\omega t+\operatorname{arg}(a-ib)]=\sqrt{a^2+b^2}\sin[\omega t+\operatorname{arg}(b+ia)].\end{aligned}\end{equation}
From this last formula, we can recover the phasor addition formula just by observing the classical trigonometric identities $\sin(\alpha \pm \beta) = \sin \alpha \cos \beta \pm \cos \alpha \sin \beta$ and $\cos(\alpha \pm \beta) = \cos \alpha \cos \beta \mp \sin \alpha \sin \beta$.
\par
There is an straightforward geometrical representation of the phasor addition formula in the euclidean case as Figure 1 shows.
\begin{figure}[h!]\label{fig2}
\begin{center}
 \includegraphics[width=.5\textwidth]{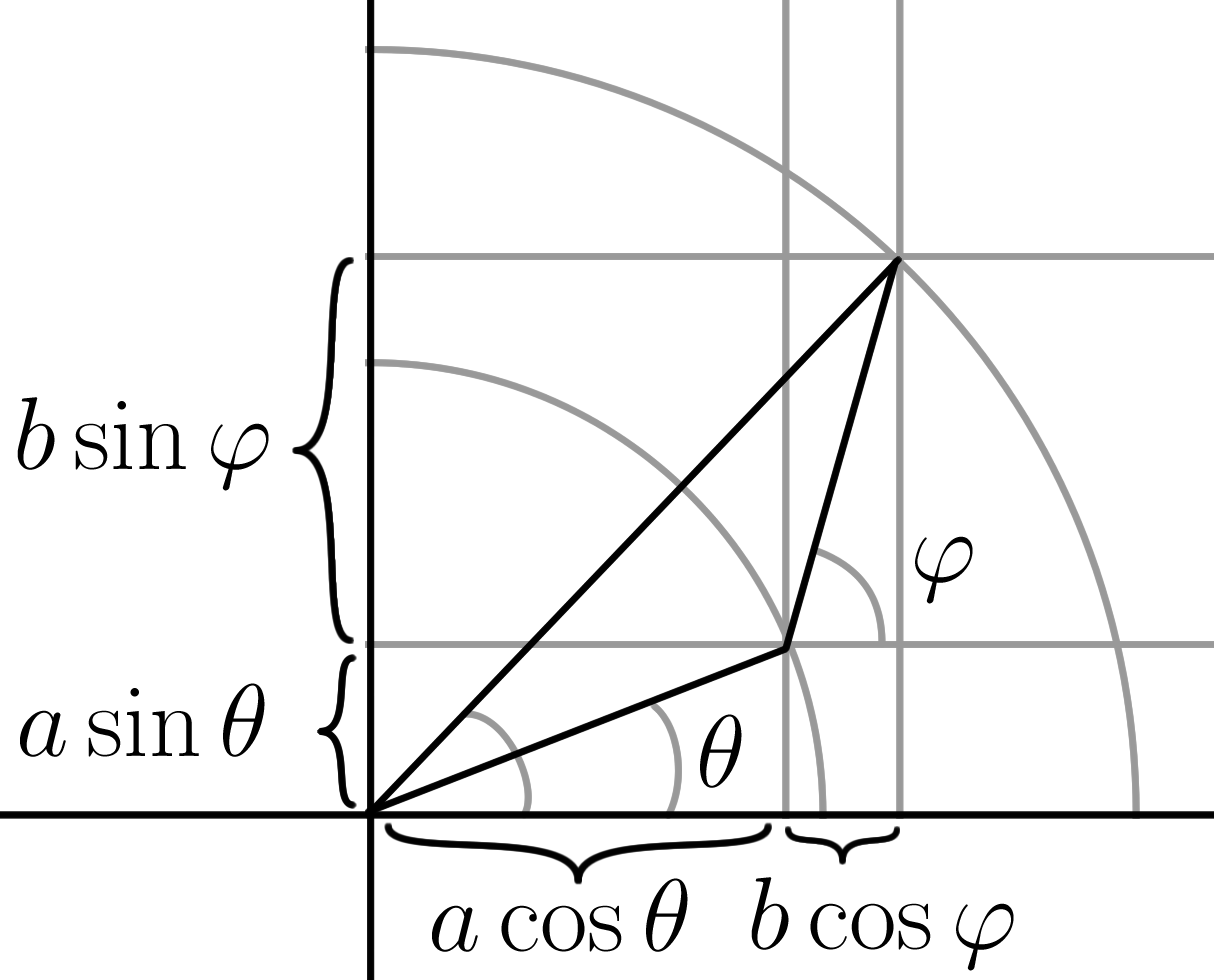}
 \end{center}
 \caption{Graphical representation of $a \cos  \theta +  b\cos \varphi$ and $a \sin  \theta +  b \sin\varphi$}
\end{figure}
The key to this graphical representation is that, on $\mathcal F_{\mathbb C}$, the sum is the sum of vectors on the plane. Then we just have to take the real part of this sum, that is, the projection onto the $OX$ axis, to obtain the desired result.
\section{The hyperbolic version of the phasor addition formula.}
We now obtain a hyperbolic counterpart of the phasor addition formula as expressed in equation \eqref{eqpafc}.\par
Let $$\mathcal G:=\{f:\mathbb R\to\mathbb R\ :\ f(t)=a\cosh\omega t+b\sinh\omega t;\ a,b\in\mathbb R\}.$$
It is straightforward to check that $(\mathcal G,+)$ is a group (and a $2$-dimensional real vector space).
Taking into account the identities
\begin{align*}
  \cosh (x + y) &= \sinh x \sinh y + \cosh x \cosh y, \\
  \sinh (x + y) &= \cosh x \sinh y + \sinh x \cosh y,
\end{align*}
it is clear that
\begin{align*} & a\cosh(\omega t+\varphi)+b\sinh(\omega t+\psi)\\ = & (a\cosh\varphi+b\sinh\psi)\cosh\omega t+(a\sinh\varphi+b\cosh\psi)\sinh\omega t\in\mathcal G.
\end{align*}
It is also clear that
\begin{align*} & a\cosh(\omega t+\varphi)+b\cosh(\omega t+\psi)\\ = & (a\cosh\varphi+b\cosh\psi)\cosh\omega t+(a\cosh\varphi+b\cosh\psi)\sinh\omega t\in\mathcal G,
\end{align*}
and
\begin{align*} & a\sinh(\omega t+\varphi)+b\sinh(\omega t+\psi)\\ = & (a\sinh\varphi+b\sinh\psi)\cosh\omega t+(a\sinh\varphi+b\sinh\psi)\sinh\omega t\in\mathcal G.
\end{align*}
\par
So we can reduce the general sums $a\cosh(\omega t+\varphi)+b\sinh(\omega t+\psi)$, $a\cosh(\omega t+\varphi)+b\cosh(\omega t+\psi)$ and $a\sinh(\omega t+\varphi)+b\sinh(\omega t+\psi)$ to the more simple case of $\alpha \cosh  \omega t +  \beta\sinh \omega t$.\par
Now we prove the following hyperbolic version of the phasor addition formula.
\begin{lem}\label{hyppha}
Let $ a$, $ b$, $t\in \mathbb R$. Then
\begin{equation}\label{hvpaf} a \cosh  \omega t +  b\sinh \omega t= \begin{cases} \sqrt{| a^2- b^2|} \cosh\left(\frac{1}{2}\ln\left|\frac{ a+ b}{ a- b}\right|+ \omega t\right) & \text{if}\quad  a>| b|,
\\ -\sqrt{| a^2- b^2|}  \cosh\left(\frac{1}{2}\ln\left|\frac{ a+ b}{ a- b}\right|+ \omega t\right) & \text{if}\quad-  a>| b|,
\\
\sqrt{| a^2- b^2|} \sinh\left(\frac{1}{2}\ln\left|\frac{ a+ b}{ a- b}\right|+ \omega t\right) & \text{if}\quad  b>| a|,
\\
-\sqrt{| a^2- b^2|} \sinh\left(\frac{1}{2}\ln\left|\frac{ a+ b}{ a- b}\right|+ \omega t\right) & \text{if }\quad - b>| a|,
\\
 a\,e^ {\omega t} & \text{if }\quad a= b,\\
 a\,e^{- \omega t} & \text{if }\quad a=- b.\end{cases}\end{equation}
\end{lem}
\begin{proof} For convenience, let $c=e^{\omega t}$. We prove the case $a>|b|$. The rest of the cases are proved in an analogous fashion.\par
Observe that, if $a>|b|$, then $a+b,a-b>0$. Thus,
\begin{align*}
 & a \cosh  \omega t +  b\sinh \omega t \\ = & \frac{a}{2}(c+c^{-1})+\frac{b}{2}(c-c^{-1})=\frac{a+b}{2}c+\frac{a-b}{2}c^{-1}\\= &\frac{\sqrt{a^2-b^2}}{2}\left(\sqrt{\frac{a+b}{a-b}}c+\sqrt{\frac{a-b}{a+b}}c^{-1}\right)=\frac{\sqrt{a^2-b^2}}{2}\left(e^{\ln\sqrt{\frac{a+b}{a-b}}}c+e^{-\ln\sqrt{\frac{a+b}{a-b}}}c^{-1}\right)\\
= & \frac{\sqrt{a^2-b^2}}{2}\left(e^{\frac{1}{2}\ln\frac{a+b}{a-b}}c+e^{-\frac{1}{2}\ln\frac{a+b}{a-b}}c^{-1}\right)\\= & \frac{\sqrt{a^2-b^2}}{2}\left(e^{\frac{1}{2}\ln\frac{a+b}{a-b}+\omega t}+e^{-\left(\frac{1}{2}\ln\frac{a+b}{a-b}+\omega t\right)}\right)=
\sqrt{a^2-b^2}\cosh\left(\frac{1}{2}\ln\frac{a+b}{a-b}+\omega t\right).
\end{align*}
\end{proof}
\begin{rem} One of the crucial differences between the hyperbolic and euclidean cases is that in the hyperbolic case there is not periodicity\footnote{Not, at least, when we consider those functions as defined on the real numbers. Hyperbolic functions are periodic when defined on the complex plane.}, what is more, we cannot relate the hyperbolic sine and cosine by a phase displacement, which implies that we may or may not be able to express an element of $\mathcal G$ in the form of a hyperbolic cosine depending on the values of $a$ and $b$, as Lemma \ref{hyppha} shows.\par
Also, comparing it with formula \eqref{eqpafc}, we observe two common elements. First, the argument of the function (euclidean or hyperbolic) involved is $\omega t$ plus a displacement depending on the parameters $a$ and $b$. The second similitude is that, multiplying such function, there is a metric applied to the vector $(a,b)$. In the euclidean case case, it is just the euclidean norm $\|(a,b)\|=\sqrt{a^2+b^2}$, that is, the square root of the metric $\mu(a,b)=a^2+b^2$ on $\mathbb R^2$. In the hyperbolic case, however, we have what is called the \textbf{Minkowski norm} $\|(a,b)\|_M=\sqrt{|\nu(a,b)|}$ where $\nu(a,b)=a^2-b^2$ is the \textbf{Minkowski metric} on $\mathbb R^2$ of signature $(1,-1)$. The Minkowski norm is not a norm in the usual sense (it is not subadditive), but it provides a useful generalization of the concept of `length' in the Minkowski plane\footnote{For more information on this topic, the book \cite{Cat} has a whole chapter on the trigonometry of the Minkowski plane.}.
\par The vectors $w=(a,b)$ are called \textbf{timelike} when $\nu(w,w)<0$, \textbf{spacelike} when $\nu(w,w)>0$ and \textbf{null}, or \textbf{lightlike} when $\nu(w,w)=0$. Observe that the two first cases of equation \eqref{hvpaf} are for spacelike vectors, the two following ones for timelike vectors, and the two last ones for lightlike vectors.
\end{rem}
It is also possible to give a geometrical representation of linear combinations of hyperbolic sines and cosines but, due to the euclidean nature of the plane, it is not as straightforward as in the euclidean case. In Figure 2 we illustrate how $a \cosh  u +  b\sinh u$ can be computed graphically.\par
\begin{figure}[h!]\label{fig1}
\begin{center}
 \includegraphics[width=.7\textwidth]{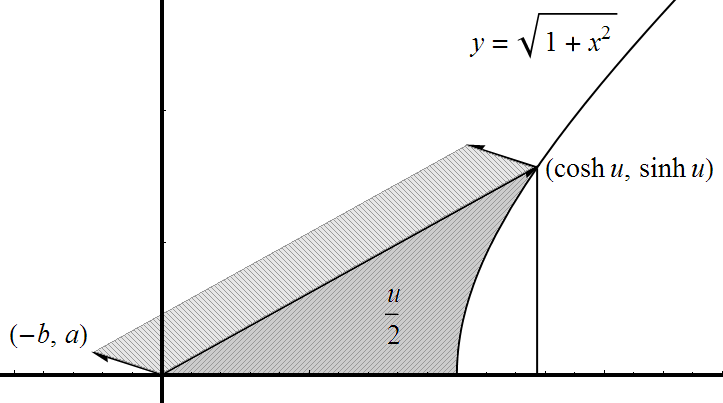}
 \end{center}
 \caption{Graphical representation of $a \cosh  u +  b\sinh u$}
\end{figure}
Consider $a,b,u>0$. The graph of the hyperbola $y^2-x^2=1$ satisfies that its points are of the form $(\cosh u,\sinh u)$. Furthermore, the area between the vector $(\cosh u,\sinh u)$, the hyperbola and the $OX$ axis is half the \textbf{hyperbolic angle} $u$. Now, if we draw the vector $(-b,a)$ and consider the parallelogram formed by the vectors $(\cosh u,\sinh u)$ and $(-b,a)$, the area of this parallelogram is precisely $a \cosh  u +  b\sinh u$. The reason for this is given by the cross product formula for the area of the parallelogram and the fact that $u>0$:
\begin{align*} & |(\cosh u,\sinh u,0)\times(-b,a,0)|=|(0,0,a\cosh u+b\sinh u)|\\ = & |a\cosh u+b\sinh u|=a\cosh u+b\sinh u.\end{align*}

\section{A final note: extending the formula}
If there is anything powerful behind the concept of exponential, hyperbolic sine, hyperbolic cosine, and other trigonometric functions, it is their wide range of definition. By this, we mean that they are defined in any Banach algebra with unity\footnote{A Banach algebra $\mathcal{A}$ is just an algebra endowed with a norm $\|\cdot\|$ that makes it a Banach space such that $\|xy\|\le\|x\|\|y\|$ for every $x,y\in\mathcal{A}$.}. Let $\mathcal{A}$ be a Banach algebra and $x\in\mathcal{A}$. We define, as usual,
\begin{align*}e^x:= & \sum_{k=0}^\infty\frac{ x^k}{k!},\\
\cosh x:= & \frac{e^x+e^{-x}}{2}=\sum_{k=0}^\infty \frac{ x^{2k}}{(2k)!},\\
\sinh x:= & \frac{e^x-e^{-x}}{2}=\sum_{k=0}^\infty \frac{ x^{2k+1}}{(2k+1)!}.\\
\end{align*}
Clearly, $\cosh$ is just the even part of the exponential and $\sinh$ its odd part, so $e^x=\cosh x+\sinh x$. If we go back to the proof of Lemma \ref{hyppha}, we observe that it relies only on these kind of definitions, so it is valid for every $a,b\in\mathbb{R}$ and any $\gamma=\omega t$ in a real Banach algebra with unity $\mathcal{A}$, in particular for $\gamma\in\mathbb C$. This is consistent with the euclidean phasor addition formula as we show next. Let $a,b,x\in\mathbb R$, assume, for instance, $a>|b|$ and consider $a\cosh ix+b\sinh ix$. Then, using Lemma \ref{hyppha},
\begin{align*}
& a \cosh  ix +  b\sinh ix=\sqrt{ a^2- b^2} \cosh\left(\frac{1}{2}\ln\frac{ a+ b}{ a- b}+ ix\right)\\
= & \frac{\sqrt{ a^2- b^2}}{2}\left(\sqrt{\frac{ a+ b}{ a- b}}e^{ix}+ \sqrt{\frac{ a- b}{ a+ b}}e^{-ix}\right) \\ = & \frac{1}{2}\left[( a+ b)(\cos x+i\sin x)+( a- b)(\cos x-i\sin x)\right]
=  a\cos x+ib\sin x
\end{align*}
which is expected from the known fact that $\cosh  ix=\cos x$, $\sinh ix=i\sin x$.\par
This observation relating the generality of the definitions of the trigonometric functions suggests yet another question. Is there a way to derive the hyperbolic phasor addition formula in the same way we derived it for the euclidean case? Or, to be more precise, is there a Banach algebra which would fulfill the role $\mathbb C$ played in the euclidean case? The answer is yes.
Remember the traditional definition of the complex numbers:
$$\mathbb C=\{x+iy\ :\ x,y\in\mathbb R,\ i\not\in\mathbb R,\ i^2=-1\}.$$
In the same way, we can define the hyperbolic numbers\footnote{See \cite{Ant,Cat} for an extended description on hyperbolic number arithmetic, calculus and geometry. It is also interesting to point out that hyperbolic numbers are a natural setting for the Theory of Relativity.}:
$$\mathbb D=\{x+jy\ :\ x,y\in\mathbb R,\ j\not\in\mathbb R,\ j^2=1\}.$$
As in the case of the complex numbers, the arithmetic in $\mathbb D$ is the natural extension assuming the distributive, associative, and commutative properties for the sum and product. Several definitions appear in a natural way, parallel to the case of $\mathbb C$. Let $w\in\mathbb D$, with $w=x+jy$. Hence
$$\overline{w}: =x-jy,\quad \Re(w):=x,\quad \Im(w):= y,$$
and since $w\overline w=x^2-y^2\in\mathbb R$, we can define
$$|w|:=\sqrt{|w\overline w|},$$
which is \emph{precisely} the Minkowski norm. It follows that $|w_1w_2|=|w_1||w_2|$ for every $w_1,w_2\in\mathbb D$ and, if $|w|\ne0$, then $w^{-1}=\overline w/|w|^2$. If we define
$$\|w\|=\sqrt{2(x^2+y^2)},$$
we have that $\|\cdot\|$ is a norm and $(\mathbb D, \|\cdot\|)$ is a Banach algebra, so the exponential and the hyperbolic trigonometric functions are well defined. Also, it is clear from the definitions that
$$e^{jw}=\cosh w+j\sinh w$$
and $|e^{jx}|=1$ for $x\in\mathbb R$.\par
The only important difference with respect to $\mathbb C$ is that $\mathbb D$ is not a division algebra (not every non-zero element has an inverse).\par
Now, let $a,b\in\mathbb R$ and $\gamma=\gamma_1+j\gamma_2\in\mathbb D$ with $\gamma_1,\gamma_2\in\mathbb R$. Observe that
$$\Re([a+jb]e^{j\gamma})=a\cosh \gamma+b\sinh \gamma.$$
We try, as we do with complex numbers, to rewrite $(a+jb)e^{j\gamma}$ as $r\,e^{j\theta}$, where $r\in\mathbb [0,+\infty)$ and $\theta\in\mathbb R$. Assume $|a+jb|\ne 0$. Then
$$r=|(a+jb)e^{j\gamma}|=|a+jb|e^{\gamma_2}$$
and
\begin{align*}(a+jb)e^{j\gamma}= & e^{\gamma_2}[a\cosh \gamma_1+b\sinh \gamma_1+j(a\sinh \gamma_1+b\cosh \gamma_1)]\\=& |a+jb|e^{\gamma_2}\cosh \theta+j|a+jb|e^{\gamma_2}\sinh \theta=r\,e^{j\theta}.
\end{align*}
Therefore,
$$a\cosh \gamma_1+b\sinh \gamma_1=|a+jb|\cosh \theta\quad\text{and}\quad b\cosh \gamma_1+a\sinh \gamma_1=|a+jb|\sinh \theta.$$
That is, assuming $a>|b|$ and defining $\sigma=\operatorname{arctanh}(b/a)$,
$$\tanh \theta=\frac{b\cosh \gamma_1+a\sinh \gamma_1}{a\cosh \gamma_1+b\sinh \gamma_1}=\frac{\frac{b}{a}+\tanh \gamma_1}{1+\frac{b}{a}\tanh \gamma_1}=\tanh(\sigma+\gamma_1)$$
so
\begin{align*}\theta & =\operatorname{arctanh}\frac{b}{a}+\gamma_1=\frac{1}{2}\ln \frac{ 1+\frac{b}{a}}{ 1-\frac{b}{a}}+\gamma_1=\frac{1}{2}\ln\frac{ a+b}{ a-b}+\gamma_1.\end{align*}
Hence,
\begin{align*}a\cosh \gamma+b\sinh \gamma= & |a+jb|e^{\gamma_2}\Re \left(e^{j\left(\frac{1}{2}\ln\frac{a+b}{a-b}+\gamma_1\right)}\right)\\= & |a+jb|e^{\gamma_2}\cosh\left(\frac{1}{2}\ln\frac{a+b}{a-b}+\gamma_1\right).\end{align*}
For $\gamma\in\mathbb R$, we recover the first case of Lemma \ref{hyppha}.


\begin{thebibliography}{99}
\bibitem{Ant} F. Antonuccio, Semi-Complex Analysis and Mathematical Physics, (2008), avaliable at \url{http://arxiv.org/pdf/gr-qc/9311032.pdf}.
\bibitem{Cab}A. Cabada, F. A. F. Tojo, Solutions of the first order linear equation with reflection and general linear conditions, \textit{Global Journal of Mathematical Sciences} \textbf{2--1} (2013) 1--8.
\bibitem{Cat} F. Catoni, D. Boccaletti, R. Cannata, V. Catoni, E. Nichelatti, P. Zampetti,  \textit{Mathematics of Minkowski Space-Time}, Birkh\"auser Verlag, Basel,  2008.
\bibitem{Des} C. A. Desoer, E. S. Kuh, \textit{Basic Circuit Theory}, McGraw-Hill, 1969.
\bibitem{Org} E. I. Organick,  \textit{A FORTRAN IV Primer},  Addison-Wesley, 1966.
\bibitem{Pha} A. G. Phadke, J. S. Thorp,   \textit{Synchronized Phasor Measurements and Their Applications}, Springer, 2008.
\bibitem{Ser}  R. A. Serway, J. W. Jewett,  \textit{Physics for Scientists and Engineers with Modern Physics}, Thomson, 2008.
\bibitem{Sta} H. E. Stanley,  \textit{Introduction to phase transitions and critical phenomena}, International Series of Monographs on Physics (Book 46). Oxford University Press, 1971.
\end{thebibliography}
\end{document}